\documentclass{amsart}

\usepackage{amsmath,amsthm,amsfonts,amssymb,amscd,latexsym}

\theoremstyle{plain}
\newtheorem{lem}{Lemma}
\newtheorem{cor}{Corollary}
\newtheorem{prop}{Proposition}
\newtheorem{thm}{Theorem}
\newtheorem*{thmw}{Theorem}

\DeclareMathOperator{\ad}{ad}
\DeclareMathOperator{\id}{id}
\DeclareMathOperator{\Ker}{Ker}
\DeclareMathOperator{\im}{Im}
\DeclareMathOperator{\Hom}{Hom}
\DeclareMathOperator{\soc}{Soc}
\DeclareMathOperator{\Der}{Der}
\DeclareMathOperator{\Zen}{Z}
\DeclareMathOperator{\Cent}{Cent}
\DeclareMathOperator{\T}{T}

\newcommand{\F}{\mathbb{F}}
\newcommand{\h}{\mathfrak{h}}
\newcommand{\tD}{\tilde{D}}
\newcommand{\caB}{\mathcal{B}}

\newcommand{\argu}{\hbox to 7truept{\hrulefill}}

\begin{document}

%%%%%%%%%%%%%%%%%%%%%%%%%%%%%%%%%%%%%%%%%%%%%%%%%%%%%%%%%%%%%%%%%%%%%%%%%%%%%%%%%%%%%%%%%%%%%%%%%%%

\title[Nilpotent outer restricted derivations]{Outer restricted derivations
of nilpotent restricted Lie algebras}

\author{J\"org Feldvoss}
\address{Department of Mathematics and Statistics, University of South Alabama,
Mobile, AL 36688--0002, USA}
\email{jfeldvoss@jaguar1.usouthal.edu}

\author{Salvatore Siciliano}
\address{Dipartimento di Matematica ``E.\ de Giorgi", Universit\`a del Salento,
Via Provinciale Lecce-Arnesano, I-73100 Lecce, Italy}
\email{salvatore.siciliano@unisalento.it}

\author{Thomas Weigel}
\address{Dipartimento di Matematica e Applicazioni, Universit\`a di Milano-Bicocca,
Via R.\ Cozzi 53, I-20125 Milano, Italy}
\email{thomas.weigel@unimib.it}



%%%%%%%%%%%%%%%%%%%%%%%%%%%%%%%%%%%%%%%%%%%%%%%%%%%%%%%%%%%%%%%%%%%%%%%%%%%%%%%%%%%%%%%%%%%%%%%%%%%

\begin{abstract}
In this paper we prove that every finite-dimensional nilpotent restricted
Lie algebra over a field of prime characteristic has an outer restricted
derivation whose square is zero unless the restricted Lie algebra is a
torus or it is one-dimensional or it is isomorphic to the three-dimensional
Heisenberg algebra in characteristic two as an ordinary Lie algebra. This
result is the restricted analogue of a result of T\^og\^o on the existence
of nilpotent outer derivations of ordinary nilpotent Lie algebras in arbitrary
characteristic and the Lie-theoretic analogue of a classical group-theoretic
result of Gasch\"utz on the existence of $p$-power automorphisms of $p$-groups.
As a consequence we obtain that every finite-dimensional non-toral nilpotent
restricted Lie algebra has an outer restricted derivation.
\end{abstract}

%%%%%%%%%%%%%%%%%%%%%%%%%%%%%%%%%%%%%%%%%%%%%%%%%%%%%%%%%%%%%%%%%%%%%%%%%%%%%%%%%%%%%%%%%%%%%%%%%%%
      
\date{}
          
\maketitle

%%%%%%%%%%%%%%%%%%%%%%%%%%%%%%%%%%%%%%%%%%%%%%%%%%%%%%%%%%%%%%%%%%%%%%%%%%%%%%%%%%%%%%%%%%%%%%%%%%%

\section{Introduction} 

%%%%%%%%%%%%%%%%%%%%%%%%%%%%%%%%%%%%%%%%%%%%%%%%%%%%%%%%%%%%%%%%%%%%%%%%%%%%%%%%%%%%%%%%%%%%%%%%%%%

In 1966 W.\ Gasch\"utz proved the following celebrated result:

\begin{thmw}{\rm (W.\ Gasch\"utz \cite{G})}
Every finite $p$-group of order $>p$ has an outer automorphism of $p$-power
order.
\end{thmw}

Since every finite nilpotent group is a direct product of its Sylow $p$-subgroups,
the outer automorphism group of a finite nilpotent group is a direct product
of the outer automorphism groups of its Sylow $p$-subgroups. Therefore it
is a direct consequence of Gasch\"utz' theorem that every finite nilpotent
group of order greater than $2$ has an outer automorphism. This answers a question
raised by E.\ Schenkman and F.\ Haimo in the affirmative (see \cite{S}).
  
Since groups and Lie algebras often have structural properties in common, it
seems rather natural to ask whether an analogue of Gasch\"utz' theorem holds
in the setting of ordinary or restricted Lie algebras. In the case of ordinary
Lie algebras a stronger version of such an analogue is already known and was
established by S.~T\^og\^o around the same time as Gasch\"utz proved his result.

\begin{thmw}{\rm (S.\ T\^og\^o \cite[Corollary 1]{T})}
Every nilpotent Lie algebra of finite dimension $>1$ over an arbitrary field
has an outer derivation whose square is zero.
\end{thmw}

In fact, T\^og\^o's result is more general (see \cite[Theorem 1]{T}) and is
a refinement of a theorem of E.\ Schenkman that establishes the existence
of outer derivations for non-zero finite-dimensional nilpotent Lie algebras
(see \cite[Theorem 4]{J2}). Much later the first author proved a restricted
analogue of Schenkman's result for $p$-unipotent restricted Lie algebras (see
\cite[Corollary 5.2]{F}). (Here we follow \cite[Section I.4, Exercise 23,
p.\ 97]{B} by calling a restricted Lie algebra $(L,[p])$ {\em $p$-unipotent\/}
if for every $x\in L$ there exists some positive integer $n$ such that $x^{[p]^n}
=0$.) In this paper we prove that every finite-dimensional nilpotent restricted
Lie algebra $L$ over a field $\F$ of characteristic $p>0$ has an outer restricted
derivation whose square is zero unless: (1) $L$ is a torus, or (2) $\dim_\F L=1$,
or (3) $p=2$ and $L$ is isomorphic to the three-dimensional Heisenberg algebra
$\h_1(\F)$ over $\F$ as an ordinary Lie algebra (see Theorem \ref{nilpoutder}).
Indeed, in these three cases every nilpotent restricted derivation is inner. As
a consequence we also obtain a generalization of \cite[Corollary 5.2]{F} to
non-toral nilpotent restricted Lie algebras (see Theorem \ref{outder}) which
is the full analogue of Schenkman's result for nilpotent restricted Lie algebras.

In the following we briefly recall some of the notation that will be used in
this paper. A derivation $D$ of a restricted Lie algebra $(L,[p])$ is called
{\em restricted\/} if $D(x^{[p]})=(\ad_Lx)^{p-1}(D(x))$ for every $x\in L$ (see
\cite[Section 4, (15), p.\ 21]{J1}) and the set of all restricted derivations
of $L$ is denoted by $\Der_p(L)$. Observe that $\Der_p(L)$ is a restricted Lie
algebra (see \cite[Theorem 4]{J1}) and that every inner derivation of $L$ is
restricted. In this paper we will use frequently without further explanation
that the vector space of outer restricted derivations $\Der_p(L)/\ad(L)$ of
$L$ is isomorphic to the {\em first adjoint restricted cohomology space\/}
$H_*^1(L,L)$ of $L$ in the sense of Hochschild (see \cite[Theorem 2.1]{H1}).
More generally, we will need for any restricted $L$-module $M$ the vector
space $$Z_*^1(L,M):=\{\,D\in\Der(L,M)\mid\forall\,x\in L:D(x^{[p]})=(x)_M^{p-1}
(D(x))\,\}$$ of {\em restricted $1$-cocycles\/} of $L$ with values in $M$ (see
again \cite[Theorem 2.1]{H1}) and the vector space $$B_*^1(L,M):=\{\,D\in\Hom_\F
(L,M)\mid\exists\,m\in M\,\forall\,x\in L:D(x)=x\cdot m\,\}$$ of {\em restricted
$1$-coboundaries\/} of $L$ with values in $M$. Then $$H_*^1(L,M):=Z_*^1(L,M)/B_*^1
(L,M)$$ denotes the {\em first restricted cohomology space\/} of $L$ with values
in $M$. For any $L$-module $M$ we denote by $M^L:=\{\,m\in M\mid\forall\,x\in L:
x\cdot m=0\,\,\}$ the {\em space of $L$-invariants\/} of $M$ and by $\soc_L(M)$
the largest semisimple $L$-submodule of $M$. Moreover, $u(L)$ denotes the {\em
restricted universal enveloping algebra\/} of $L$ (see \cite[Section 2]{J1},
\cite[Section V.7, Theorem 12]{J3}, or \cite[Section 2.5]{SF}). For a subset $S$
of $L$ we denote by $\langle S\rangle_p$ the {\em $p$-subalgebra\/} of $L$ {\em
generated by\/} $S$ and by $\Cent_L(S)$ the {\em centralizer\/} of $S$ in $L$.
Finally, $\Zen(L):=\Cent_L(L)$ is the {\em center\/} of $L$ and $L^\prime$ is the
{\em derived subalgebra\/} of $L$. For more notation and well-known results from
the structure theory of ordinary and restricted Lie algebras we refer the reader
to the first two chapters in \cite{SF}.

%%%%%%%%%%%%%%%%%%%%%%%%%%%%%%%%%%%%%%%%%%%%%%%%%%%%%%%%%%%%%%%%%%%%%%%%%%%%%%%%%%%%%%%%%%%%%%%%%%%

\section{Preliminaries}

%%%%%%%%%%%%%%%%%%%%%%%%%%%%%%%%%%%%%%%%%%%%%%%%%%%%%%%%%%%%%%%%%%%%%%%%%%%%%%%%%%%%%%%%%%%%%%%%%%%

\subsection{Restricted derivations and cohomology}

Let $L$ be a restricted Lie algebra over a field of characteristic $p>0$ and let
$I$ be a $p$-ideal of $L$. Then the centralizer $\Cent_L(I)=L^I$ of $I$ in $L$ is
a $p$-ideal of $L$. By virtue of the five-term exact sequence associated to the
Hochschild-Serre spectral sequence the canonical mapping $$\iota:H^1_\ast(L/I,\Cent_L
(I))\longrightarrow H^1_\ast(L,L)$$ is injective. Hence one has the following sufficient
condition for the existence of outer restricted derivations:

\begin{lem}\label{out}
Let $L$ be a restricted Lie algebra over a field of characteristic $p>0$, let $I$
be a $p$-ideal of $L$, and let $D$ be a restricted derivation of $L$ satisfying
$I\subseteq\Ker(D)$ as well as $\im(D)\subseteq\Cent_L(I)$, but $\im(D)\not
\subseteq[L,\Cent_L(I)]$. Then $D$ is not inner.
\end{lem}

\begin{proof}
Let $\tD\in Z^1_\ast(L/I,\Cent_L(I))$ be the linear transformation induced by $D$.
By hypothesis, $\im(\tD)\not\subseteq[L,\Cent_L(I)]$ and thus $\tD\not\in B^1_\ast
:=B^1_\ast(L/I,\Cent_L(I))$. Hence $\tD+B^1_\ast\not=0$, and therefore $D+\ad(L)=
\iota(\tD+B^1_\ast)\not=0$, so that $D$ is not inner.
\end{proof}

By using the injectivity of $\iota$ again, one deduces the following criterion for
the existence of outer restricted derivations whose square is zero.

\begin{lem}\label{good}
Let $L$ be a restricted Lie algebra over a field of characteristic $p>0$ that contains
$p$-ideals $I$ and $J$ with the following properties:
\begin{itemize}
\item[(i)]  $J\subseteq I\cap\Cent_L(I)=\Zen(I)$ and
\item[(ii)] the canonical mapping $\gamma:H^1_\ast(L/I,J)\to H^1_\ast(L/I,\Cent_L(I))$
            is non-trivial.
\end{itemize}
Then there exists an outer restricted derivation of $L$ whose square is zero.
\end{lem}

\begin{proof}
Let $\tD\in Z^1_\ast(L/I,J)$ be such that $\gamma(\tD+B^1_\ast(L/I,J))\not=0$ and let
$D\in Z^1_\ast(L,L)$ be the linear transformation induced by $\tD$. Then it follows
from $$\im(D)\subseteq J\subseteq I\subseteq\Ker(D)$$ that $D^2=0$. Moreover, $$D+\ad
(L)=\iota(\gamma(\tD+B^1_\ast(L/I,J)))\not=0\,,$$ and thus $D$ is an outer restricted
derivation of $L$ with $D^2=0$.
\end{proof}

\subsection{Abelian $p$-ideals}

In the proof of our main result we will need the following criterion for the existence
of certain $p$-supplements of abelian $p$-ideals (for the first part of the proof see
also \cite[Lemma 3.1]{H1}).

\begin{prop}\label{split}
Let $L$ be a restricted Lie algebra over a field of characteristic $p>0$ and let $A$
be an abelian $p$-ideal of $L$ containing $\Zen(L)$. If $H^2_\ast(L/A,A)=0$, then there
exists a $p$-subalgebra $H$ of $L$ such that $L=A+H$ and $A\cap H=\Zen(L)$.
\end{prop}

\begin{proof} 
Let $\caB_A$ be a vector space basis of $A$ being contained in a vector space basis
$\caB$ of $L$. According to \cite[Section V.7, Theorem 11]{J3} or \cite[Theorem 2.2.3]{SF},
the function $\psi:\caB\to L$ defined by
\begin{eqnarray*}
\psi(x):=
\left\{
\begin{array}{cl}
0 & \mbox{if }x\in\caB_A\\ 
x^{[p]} & \mbox{if }x\in\caB\setminus\caB_A
\end{array}
\right.
\end{eqnarray*}
can be extended uniquely to a $p$-mapping $\argu^{[p]^\prime}:L\to L$ making $(L,[p]^\prime)$
a restricted Lie algebra. By construction,
\begin{equation}\label{pdif}
x^{[p]}-x^{[p]^\prime}\in\Zen(L)\quad\mbox{for every }x\in L\,.
\end{equation}
The ideal $A$ -- with trivial $p$-mapping -- is also a $p$-ideal of $(L,[p]^\prime)$,
and the identity mapping induces an isomorphism between $(L/A,[p])$ and $(L/A,[p]^\prime)$.
It follows from $H^2_\ast(L/A,A)=0$ that there exists a $p$-subalgebra $C$ of $(L,[p]^\prime)$
such that $L=A\oplus C$ (see \cite[Theorem 3.3]{H1}). Then by \eqref{pdif}, $H:=\Zen(L)
\oplus C$ is a $p$-subalgebra of $(L,[p])$ with the desired properties.
\end{proof}

\noindent {\bf Remark.}
A similar proof shows that Proposition \ref{split} holds more generally for abelian
$p$-ideals $A$ of $L$ that do not necessarily contain the center of $L$ if one replaces
everywhere $\Zen(L)$ by the image of $A$ under the $p$-mapping of $(L,[p])$.
\bigskip

The next result shows that maximal abelian $p$-ideals of finite-dimensional non-abelian
nilpotent restricted Lie algebras are self-centralizing and contain the center properly.

\begin{prop}\label{maxab}
Let $L$ be a finite-dimensional non-abelian nilpotent restricted Lie algebra over a
field of characteristic $p>0$ and let $A$ be a maximal abelian $p$-ideal of $L$. Then
$\Zen(L)\subsetneq A=\Cent_L(A)$.
\end{prop}

\begin{proof}
Suppose by contradiction that $C:=\Cent_L(A)\supsetneq A$. As $C/A$ is a non-zero ideal
of the nilpotent Lie algebra $L/A$, there is an element $x$ in $L$ that does not belong
to $A$ such that $x+A\in\Zen(L/A)\cap C/A$. Put $X:=A+\langle x\rangle_p$. Then $X$ is
an abelian $p$-ideal of $L$ properly containing $A$ which by the maximality of $A$ implies
that $L=X$ is abelian contradicting the hypothesis that $L$ is non-abelian.

It follows from $A=\Cent_L(A)$ that $A$ is a faithful $L/A$-module under the induced
adjoint action. Suppose now that $\Zen(L)=A$. Then $\Zen(L)$ is a faithful and trivial
$L/\Zen(L)$-module under the induced adjoint action which implies that $L=\Zen(L)$ is
abelian, which again is a contradiction.
\end{proof}

Let $\T_p(L)$ denote the set of all semisimple elements of a finite-dimensional nilpotent
restricted Lie algebra $L$. Since $L$ is nilpotent, we have that $\T_p(L)\subseteq\Zen(L)$
and thus $\T_p(L)$ is the unique maximal torus of $L$. Furthermore, $\T_p(L)$ is a $p$-ideal
of $L$ and it follows from \cite[Theorem 2.3.4]{SF} that $L/I$ is $p$-unipotent for every
$p$-ideal $I$ of $L$ that contains $\T_p(L)$. These well-known results will be used in the
following proofs without further explanation.

\begin{prop}\label{bad}
Let $L$ be a finite-dimensional non-abelian nilpotent restricted Lie algebra over a
field $\F$ of characteristic $p>0$ and let $A$ be a maximal abelian $p$-ideal of $L$.
If furthermore every restricted derivation $D$ of $L$ with $D^2=0$ is inner, then $A$
is a free $u(L/A)$-module of rank $r:=\dim_\F\Zen(L)$. In  particular, $\dim_\F L=d+r
\cdot p^d$, where $d:=\dim_\F L/A$.
\end{prop}

\begin{proof}
Suppose that $H^1_\ast(L/A,A)\not=0$. Then by Lemma \ref{good} (applied to $I:=J:=A$)
in conjunction with Proposition \ref{maxab}, there exists an outer restricted derivation
of $L$ whose square is zero, a contradiction. Hence $H^1_\ast(L/A,A)=0$.

Another application of Proposition \ref{maxab} yields $\T_p(L)\subseteq A$. Consequently,
$L/A$ is $p$-unipotent and it follows from \cite[Proposition 5.1]{F} that $A$ is
a free $u(L/A)$-module. Moreover, by virtue of the Engel-Jacobson theorem (see \cite[
Corollary 1.3.2(1)]{SF}), the trivial $L/A$-module $\F$ is the only irreducible restricted
$u(L/A)$-module up to isomorphism. Thus $$\soc_{L/A}(A)=A^{L/A}=\Zen(L)\,,$$ and $A$ is
a free $u(L/A)$-module of rank $r:=\dim_\F\Zen(L)$. In particular, $\dim_\F A=r\cdot
p^d$ and the dimension formula for $L$ follows.
\end{proof}

Note that $\dim_\F A=\dim_F\Zen(L)\cdot p^{\dim_\F L/A}$ shows again that $A$ contains
$\Zen(L)$ properly if $L$ is not abelian.

\subsection{Maximal $p$-ideals of nilpotent restricted Lie algebras}

For the convenience of the reader we include a proof of the following result.

\begin{lem}\label{pnilmax}
Let $L$ be a finite-dimensional non-toral nilpotent restricted Lie algebra over
a field of characteristic $p>0$. Then every maximal $p$-ideal of $L$ containing
the unique maximal torus of $L$ has codimension one in $L$.
\end{lem}

\begin{proof}
Let $I$ be a maximal $p$-ideal of $L$ that contains the unique maximal torus
$\T_p(L)$ of $L$. Then $\Zen(L/I)\neq 0$ and because $L/I$ is $p$-unipotent,
there exists a non-zero central element $z$ of $L/I$ such that $z^{[p]}=0$. Then
$Z:=\F z$ is a one-dimensional $p$-ideal of $L/I$ and the inclusion-preserving
bijection between $p$-ideals of $L/I$ and $p$-ideals of $L$ that contain $I$ in
conjunction with the maximality of $I$ yields that $L/I=Z$ is one-dimensional.
\end{proof}

%%%%%%%%%%%%%%%%%%%%%%%%%%%%%%%%%%%%%%%%%%%%%%%%%%%%%%%%%%%%%%%%%%%%%%%%%%%%%%%%%%%%%%%%%%%%%%%%%%%

\section{Main results}

%%%%%%%%%%%%%%%%%%%%%%%%%%%%%%%%%%%%%%%%%%%%%%%%%%%%%%%%%%%%%%%%%%%%%%%%%%%%%%%%%%%%%%%%%%%%%%%%%%%

The main goal of this paper is to establish the existence of nilpotent outer
restricted derivations of finite-dimensional nilpotent restricted Lie algebras.
It is well-known that the restricted cohomology of tori vanishes (see \cite{H1}
and the main result of \cite{H2} or \cite[Corollary 3.6]{F}). Hence tori have
no outer restricted derivations. For the convenience of the reader we include
the following straightforward proof of the latter statement.

\begin{prop}\label{tori}
Every restricted derivation of a torus over any field of characteristic $p>0$
vanishes.
\end{prop}

\begin{proof}
Suppose that $L$ is a torus and let $D$ be any restricted derivation of $L$.
Then we have $D(x^{[p]})=(\ad_Lx)^{p-1}(D(x))=0$ for every $x\in L$, and
inductively, $D(x^{[p]^n})=0$ for every $x\in L$ and every positive integer
$n$. Since $x\in\langle x^{[p]}\rangle_p$, we conclude that $D(x)=0$ for every
$x\in L$.
\end{proof}

If $L$ is one-dimensional, then either $L$ is a torus or $L=\F e$ with $e^{[p]}
=0$. In the second case the (outer) restricted derivations of $L$ coincide with
the vector space endomorphisms $\mathrm{End}_\F(L)=\F\cdot\id_L$ of $L$ and
therefore no non-zero (outer) restricted derivation is nilpotent.

Indeed there is only one more finite-dimensional nilpotent restricted Lie algebra
(up to isomorphism of ordinary Lie algebras) that has no nilpotent outer restricted
derivations, namely the three-dimensional restricted Heisenberg algebra over a
field of characteristic two. Since the derived subalgebra of any Heisenberg algebra
is central, it follows from \cite[Example 2, p.\ 72]{SF} that Heisenberg algebras
are restrictable. For the proof of the next result one only needs that the image of
every $p$-mapping is central so that the result does not depend on the particular
choice of the $p$-mapping.

\begin{prop}\label{3dimheisenberg}
Every nilpotent restricted derivation of a three-dimensional restricted Heisenberg
algebra over a field of characteristic two is inner. 
\end{prop}

\begin{proof}
Let $\F$ be any field of characteristic $2$ and let $L$ be the three-dimensional
Heisenberg algebra $\h_1(\F)=\F x+\F y+\F z$ defined by $[x,y]=z\in\Zen(L)$. Since
$L^\prime$ is central, we have that
\begin{equation}\label{char2heisenberg}
L^{[2]}\subseteq\Zen(L)=\F z\,.
\end{equation}
Suppose that $D$ is any nilpotent outer restricted derivation of $L$. It follows
from $$[x,D(z)]=[y,D(z)]=0$$ that $z$ is an eigenvector of $D$. This forces $D(z)
=0$, as $D$ is nilpotent. Moreover, by (\ref{char2heisenberg}) we obtain that
\begin{align}\nonumber
(\ad_Lx)(D(x))&=D(x^{[2]})=0\,,\\ \nonumber
(\ad_Ly)(D(y))&=D(y^{[2]})=0\,,
\end{align}
and then
\begin{align}\nonumber
D(x)&=\alpha x+\lambda z\,,\\ \nonumber
D(y)&=\beta y+\mu z
\end{align}
for suitable $\alpha,\beta,\lambda,\mu\in\F$. Furthermore, $D([x,y])=0$ yields $\alpha
=\beta$. 

Now, observe that $D^2=\alpha D$ and therefore $D^n=\alpha^{n-1}D$ for every positive
integer $n$. As $D$ is nilpotent, we conclude that $\alpha=0$. Thus one has $$D=\ad_L
(\lambda y+\mu x)\,,$$ so that $D$ is inner, as claimed.
\end{proof}

Now we are ready to prove our first main result:

\begin{thm}\label{nilpoutder}
Let $L$ be a nilpotent restricted Lie algebra of finite dimension $>1$ over a field
$\F$ of characteristic $p>0$. Then $L$ has an outer restricted derivation $D$ with
$D^2=0$ unless $L$ is a torus or $p=2$ and $L\cong\h_1(\F)$ as ordinary Lie algebras.
\end{thm}

\begin{proof}
Assume that $L$ is neither a torus nor one-dimensional. In the following $T$ will
denote the unique maximal torus $\T_p(L)$ of $L$ and by assumption $L/T\neq 0$. We
proceed by a case-by-case analysis. 
\bigskip

{\bf Case 1:} \emph{There exists a maximal $p$-ideal $I$ of $L$ containing $T$ such
that $\Zen(L)\nsubseteq I$.}

\noindent According to Lemma \ref{pnilmax}, $I$ has codimension $1$ in $L$ and therefore
$L=\F x\oplus I$ for any $x\in\Zen(L)\backslash I$. Let $0\neq z\in\Zen(I)$ and consider
the linear transformation $D$ of $L$ defined by setting $D(x):=z$ and $D(y):=0$ for every
$y\in I$. Then $D$ is a derivation of $L$ (see \cite[Section I.4, Exercise 8(a), p.\ 92]{B})
with $D^2=0$. Moreover, the inner derivation $\ad_La$ vanishes on $x$ for every $a\in L$
and thus $D$ cannot be inner. Finally, as $L/I$ is $p$-unipotent and $I$ has codimension
$1$ in $L$, it follows that $L^{[p]}\subseteq I$. This implies that $D(l^{[p]})=0=(\ad_L
l)^{p-1}(D(l))$ for every $l\in L$, and so $D$ is restricted.
\bigskip

Note that Case 1 covers already the abelian case. So for the rest of the proof we may
assume that $L$ is not abelian.
\bigskip

{\bf Case 2:} \emph{$\Zen(L)\subseteq I$ for every maximal $p$-ideal $I$ of $L$ containing
$T$.}

\noindent Suppose that $\Cent_L(I)\not\subseteq I$ for some maximal $p$-ideal $I$ of $L$
containing $T$. Then there exists $x\in\Cent_L(I)\backslash I$ such that $L=\F x\oplus I$;
in particular, $x\in\Zen(L)\subseteq I$, a contradiction. Hence we may assume from now on
that $\Cent_L(I)=\Zen(I)$ for every maximal $p$-ideal $I$ of $L$ containing $T$.
\bigskip

{\bf Case 2.1:} \emph{There exists a maximal $p$-ideal $I$ of $L$ containing $T$ such that
$\Cent_L(I)$ $=\Zen(L)$.}

\noindent Let $x\in L\backslash I$ be such that $L=\F x\oplus I$. Let $0\neq z\in\Zen(L)$ and
consider the vector space endomorphism $D$ of $L$ given by $D(x):=z$ and $D(y):=0$ for every
$y\in I$. Then as above, $D$ is a restricted derivation of $L$ with $D^2=0$. By construction,
$\Ker(D)=I$ and $\im(D)=\F z$. Since by hypothesis $[L,\Cent_L(I)]=0$, it follows from Lemma
\ref{out} that $D$ is not inner.
\bigskip

{\bf Case 2.2:} \emph{$\Zen(L)\subsetneq\Cent_L(I)=\Zen(I)$ for every maximal $p$-ideal
$I$ of $L$ containing $T$.}

\noindent Suppose that every restricted derivation $D$ of $L$ satisfying $D^2=0$ is inner.
Let $A$ be a maximal abelian $p$-ideal of $L$. Since $L$ is not abelian, it follows from
Proposition \ref{maxab} that $A$ contains $\Zen(L)$ properly. According to Proposition
\ref{bad}, $A$ is a free $u(L/A)$-module of rank $r:=\dim_\F\Zen(L)$ and $\dim_\F L=d+r
\cdot p^d$, where $d:=\dim_\F L/A$.

Since $A$ is a free $u(L/A)$-module, \cite[Proposition 5.1]{F} implies that $H^2_\ast
(L/A,A)=0$. Hence by Proposition \ref{split}, there exists a $p$-subalgebra $H$ of $L$
such that $L=A+H$ and $A\cap H=\Zen(L)$. In particular, $H$ contains $T$. As $A$ contains
$Z(L)$ properly, one has $H\neq L$. Choose now a proper $p$-subalgebra $J$ of $L$ with
$H\subseteq J$ of maximal dimension. Then $J$ is a maximal $p$-subalgebra of $L$. Since
$J$ is properly contained in $L$, the nilpotency of $L$ implies that $J$ is properly
contained in the $p$-subalgebra $N_L(J):=\{\,x\in L\mid [x,J]\subseteq J\,\}$. Hence
the maximality of $J$ yields $N_L(J)=L$. Consequently, $J$ is a maximal $p$-ideal of
$L$ containing $H$ and in particular $T$. According to Lemma~\ref{pnilmax}, $J$ is of
codimension $1$ in $L$. Since $L=A+J$, one has $$\dim_\F A/A\cap J=\dim_\F A+J/J=
\dim_\F L/J=1\,,$$ so that $A\cap J$ has codimension $1$ in $A$.

Let $B:=\Cent_L(J)=\Zen(J)$. By hypothesis, there exists an element $e\in B\backslash
\Zen(L)$. We obtain from $B\subseteq\Cent_L(H)$ and $L=A+H$ that $B\cap A=\Zen(L)$
which implies that $e\not\in A$. Since $B$ is an abelian $p$-ideal of $L$, $e^{[p]}\in
\Zen(L)\subseteq A$. Hence it follows from \cite[Lemma 2.1.2]{SF} that $\F e\oplus A$
is a $p$-subalgebra of $L$.

Let $E:=\F e\oplus A/A$ be the one-dimensional $p$-unipotent $p$-subalgebra generated
by the image of $e$ in $L/A$. Note that $A\cap J\subseteq A^E$ and $\dim_\F A\cap J=r
\cdot p^d-1$. By construction, $A$ is a free $u(E)$-module of rank $r\cdot p^{d-1}$.
Consequently, $$r\cdot p^d-1=\dim_\F A\cap J\le\dim_\F A^E=\dim_\F\soc_E(A)=r\cdot
p^{d-1}\,.$$ But this implies $p^d=2$ and $r=1$. Hence $\F$ has characteristic $2$
and $\dim_\F L=3$. Since the Heisenberg algebra is the only non-abelian nilpotent
three-dimensional Lie algebra (up to isomorphism), this finishes the proof.
\end{proof}

Theorem \ref{nilpoutder} in conjunction with the introductory remarks of this section
yields a characterization of finite-dimensional nilpotent restricted Lie algebras in
terms of the existence of nilpotent outer restricted derivations.

\begin{cor}\label{char}
For a finite-dimensional nilpotent restricted Lie algebra $(L,[p])$ over a field $\F$
of characteristic $p>0$ the following statements are equivalent:
\begin{enumerate}
\item [(i)]   $L$ is a torus or $\dim_\F L=1$ and $L^{[p]}=0$ or $p=2$ and $L\cong
              \h_1(\F)$ as ordinary Lie algebras.
\item [(ii)]  $L$ has no nilpotent outer restricted derivation.
\item [(iii)] $L$ has no outer restricted derivation $D$ with $D^2=0$.
\end{enumerate}
\end{cor}

\begin{proof}
Clearly, (ii) implies (iii) and it follows from Theorem \ref{nilpoutder} that (iii)
implies (i). It remains to prove the implication (i)$\Longrightarrow$(ii) which is
an immediate consequence of Proposition \ref{tori}, the paragraph thereafter, and
Proposition \ref{3dimheisenberg}.
\end{proof}

\noindent If we exclude the one-dimensional case and assume that the characteristic
of the ground field is greater than two, then Theorem \ref{nilpoutder} can be used
to characterize the tori among nilpotent restricted Lie algebras in terms of the
non-existence of restricted derivations with various nilpotency properties.

\begin{cor}
For a nilpotent restricted Lie algebra $L$ of finite dimension $>1$ over a field of
characteristic $p>2$ the following statements are equivalent:
\begin{enumerate}
\item [(i)]   $L$ is a torus.
\item [(ii)]  $L$ has no non-zero restricted derivation.
\item [(iii)] $L$ has no nilpotent outer restricted derivation.
\item [(iv)]  $L$ has no outer restricted derivation $D$ with $D^2=0$.
\end{enumerate}
\end{cor}

\begin{proof}
Obviously, the implications (ii)$\Longrightarrow$(iii) and (iii)$\Longrightarrow$(iv)
hold. The implications (iv)$\Longrightarrow$(i) and (i)$\Longrightarrow$(ii) are
immediate consequences of Theorem \ref{nilpoutder} and Proposition \ref{tori},
respectively.
\end{proof}

Finally, we obtain from Theorem \ref{nilpoutder} the following generalization of
\cite[Corollary 5.2]{F}:

\begin{thm}\label{outder}
Every finite-dimensional non-toral nilpotent restricted Lie algebra over a field of
characteristic $p>0$ has an outer restricted derivation.
\end{thm}

\begin{proof}
According to Theorem \ref{nilpoutder} and the first paragraph after Proposition
\ref{tori}, it is enough to show the assertion for the three-dimensional restricted
Heisenberg algebra $L:=\h_1(\F)$ over a field $\F$ of characteristic $2$. Denote the
one-dimensional center of $L$ by $Z$. If $L$ is $2$-unipotent, then the claim follows
from \cite[Corollary 5.2]{F}. Suppose now that $L$ is not $2$-unipotent. Then the
center $Z$ of $L$ is a one-dimensional torus. According to \cite[Proposition 3.9]{F}
in conjunction with \cite[Corollary 3.6]{F}, we obtain that $H_*^1(L,L)\cong H_*^1
(L/Z,L)$. Since $L^\prime$ is central, we have $L^{[2]}\subseteq Z$, and consequently,
$L/Z$ is $2$-unipotent. Suppose now that $L$ has no outer restricted derivation and
therefore $H_*^1(L/Z,L)\cong H_*^1(L,L)=0$. Then an application of \cite[Proposition
5.1]{F} yields that $L$ is a free $u(L/Z)$-module and it follows that $$3=\dim_\F L
\ge\dim_\F u(L/Z)=4\,,$$ a contradiction.
\end{proof}

\noindent {\bf Remark.}
It can also be seen directly that any three-dimensional restricted Heisenberg algebra
$\h_1(\F)$ over a field $\F$ of characteristic $2$ has an outer restricted derivation.
Namely, let $D$ be a linear transformation of $\h_1(\F)$ vanishing on its center $\Zen
(\h_1(\F))$ and inducing the identity transformation on $\h_1(\F)/\Zen(\h_1(\F))$. Then
it is straightforward to see that $D$ is indeed an outer restricted derivation of $\h_1
(\F)$.
\bigskip

Neither Theorem \ref{nilpoutder} nor Theorem \ref{outder} does generalize to non-toral
{\em solvable\/} restricted Lie algebras with {\em non-zero center\/} as the direct
product of a two-dimensional non-abelian restricted Lie algebra and a one-dimensional
torus over any field of characteristic $p>0$ shows. In this case it is not difficult
to see that there are no outer restricted derivations.

%%%%%%%%%%%%%%%%%%%%%%%%%%%%%%%%%%%%%%%%%%%%%%%%%%%%%%%%%%%%%%%%%%%%%%%%%%%%%%%%%%%%%%%%%%%%%%%%%%%  

%%%%%%%%%%%%%%%%%%%%%%%%%%%%%%%%%%%%%%%%%%%%%%%%%%%%%%%%%%%%%%%%%%%%%%%%%%%%%%%%%%%%%%%%%%%%%%%%%%%

\end{document}